\documentclass[twoside,leqno,10pt, A4]{amsart}
\usepackage{amsfonts}
\usepackage{amsmath}
\usepackage{amscd}
\usepackage{amssymb}
\usepackage{amsthm}
\usepackage{amsrefs}
\usepackage{latexsym}
\usepackage{mathrsfs}
\usepackage{bbm}
\usepackage{enumerate}
\usepackage{color}
\begin{document}

\newtheorem{theorem}[subsection]{Theorem}
\newtheorem{proposition}[subsection]{Proposition}
\newtheorem{lemma}[subsection]{Lemma}
\newtheorem{corollary}[subsection]{Corollary}
\newtheorem{conjecture}[subsection]{Conjecture}
\newtheorem{prop}[subsection]{Proposition}
\numberwithin{equation}{section}
\newcommand{\mr}{\ensuremath{\mathbb R}}
\newcommand{\mc}{\ensuremath{\mathbb C}}
\newcommand{\dif}{\mathrm{d}}
\newcommand{\intz}{\mathbb{Z}}
\newcommand{\ratq}{\mathbb{Q}}
\newcommand{\natn}{\mathbb{N}}
\newcommand{\comc}{\mathbb{C}}
\newcommand{\rear}{\mathbb{R}}
\newcommand{\prip}{\mathbb{P}}
\newcommand{\uph}{\mathbb{H}}
\newcommand{\fief}{\mathbb{F}}
\newcommand{\majorarc}{\mathfrak{M}}
\newcommand{\minorarc}{\mathfrak{m}}
\newcommand{\sings}{\mathfrak{S}}
\newcommand{\fA}{\ensuremath{\mathfrak A}}
\newcommand{\mn}{\ensuremath{\mathbb N}}
\newcommand{\mq}{\ensuremath{\mathbb Q}}
\newcommand{\half}{\tfrac{1}{2}}
\newcommand{\f}{f\times \chi}
\newcommand{\summ}{\mathop{{\sum}^{\star}}}
\newcommand{\chiq}{\chi \bmod q}
\newcommand{\chidb}{\chi \bmod db}
\newcommand{\chid}{\chi \bmod d}
\newcommand{\sym}{\text{sym}^2}
\newcommand{\hhalf}{\tfrac{1}{2}}
\newcommand{\sumstar}{\sideset{}{^*}\sum}
\newcommand{\sumprime}{\sideset{}{'}\sum}
\newcommand{\sumprimeprime}{\sideset{}{''}\sum}
\newcommand{\shortmod}{\ensuremath{\negthickspace \negthickspace \negthickspace \pmod}}
\newcommand{\V}{V\left(\frac{nm}{q^2}\right)}
\newcommand{\sumi}{\mathop{{\sum}^{\dagger}}}
\newcommand{\mz}{\ensuremath{\mathbb Z}}
\newcommand{\leg}[2]{\left(\frac{#1}{#2}\right)}
\newcommand{\muK}{\mu_{\omega}}
\newcommand{\seteq}{:=}
\newcommand{\odd}{\mathrm{\ primary}}
\newcommand{\res}{\mathrm{Res}}

\newcommand{\hf}{\tfrac{1}{2}}
\newcommand{\af}{\mathfrak{a}}
\newcommand{\Wf}{\mathcal{W}}

\title[Large values of cubic and quartic Dirichlet $L$-functions of prime moduli]{Large values of cubic and quartic Dirichlet $L$-functions of prime moduli}

\author[P. Gao]{Peng Gao}
\address{School of Mathematical Sciences, Beihang University, Beijing 100191, China}
\email{penggao@buaa.edu.cn}

\author[L. Zhao]{Liangyi Zhao}
\address{School of Mathematics and Statistics, University of New South Wales, Sydney NSW 2052, Australia}
\email{l.zhao@unsw.edu.au}

\begin{abstract}
In this paper, we apply the resonator method to exhibit large values at the central point of the families of cubic and quartic Dirichlet $L$-functions of prime moduli under the generalized Riemann hypothesis. 
\end{abstract}

\maketitle



\section{Introduction}

Extreme values of $L$-functions have attracted much attention in the literature. In \cite{Sound08}, K. Soundararajan introduced the resonance method to study large values of $L$-functions. This powerful method allowed him to show that for sufficiently large $X$, 
\begin{align*}
	\max_{\substack{d \text{ fundamental discrimint }\\X<|d|\le 2X}}|L(\tfrac 12, \chi_{d})|\geq \exp\left(\left(\frac 1{\sqrt{5}}  +o(1) \right)\sqrt{\frac{\log X }{\log_2 X}}\right), 
\end{align*}
  where $\chi_d=\leg {d}{\cdot}$ denotes the Kronecker symbol. Here and throughout the paper, we denote $\log_j$ the $j$-fold iterated logarithm. \newline
  
The method of Soundararjan was subsequently utilized by C. Aistleitner \cite{Aistleitner16}, A. Bondarenko and K. Seip in \cite{BS18}, and by R. de la Bret\`eche and G. Tenenbaum in \cite{BT19} to study large values of the Riemann zeta function $\zeta(s)$.  Incorporating the long resonator method of Bondarenko--Seip in \cite{BS18}, P.  Darbar and G. Maiti  showed in \cite{DM25} that under the generalized Riemann hypothesis (GRH), for sufficiently large $X$, 
\begin{align*}
	\max_{\substack{d \text{ fundamental discrimint }\\X<|d|\le 2X}}|L(\tfrac 12, \chi_{d})|\geq \exp\left(\left(\frac 1 2 +o(1) \right)\sqrt{\frac{\log X \log_3 X}{\log_2 X}}\right).
\end{align*}

 Throughout the paper, we reserve the letter $p$ for a positive rational prime in $\mz$. It is shown in \cite{FHX26} by M. Fan, S. Hua and S. Xie using the resonance method that for sufficiently large $X$, 
\begin{align}
\label{p}
\max_{\substack{X< p \leq 2X \\ p \equiv 1 \shortmod 8}}|L(\tfrac 12, \chi_p)|\geq \exp\left(\left(\sqrt{\frac 8{45}}+o(1)\right)\sqrt{\frac{\log X }{\log_2 X}}\right).
\end{align}
Based on the ideas in \cite{DM25}, the first-named author proved in \cite{Gao2026-6} that under GRH,  for sufficiently large $X$,
\begin{align*}
\begin{split}
  \max_{\substack{X< p \le 2X}}\Big|L(\tfrac 12, \chi_{8p})\Big |\geq \exp\left(\left(\frac 1 2 +o(1) \right)\sqrt{\frac{\log X \log_3 X}{\log_2 X}}\right).
\end{split}
\end{align*}
As noted in \cite{DM25}, one may adapt the approaches therein to show that the above inequality remains valid with $L(\frac 12, \chi_{8p})$ replaced by $L(\frac 12, \chi_{p})$ for a fundamental discriminant $p$, leading to an improvement of \eqref{p} under GRH. \newline
   
For any $p$, let $\chi_0$ denote the principle Dirichlet character modulo $p$ (we omit writing out its dependence on $p$ as it will be clear from the context). We say a Dirichelt character $\chi$ modulo $p$ is a cubic (resp. quartic) character if $\chi^3 = \chi_0$ (resp. $\chi^4 = \chi_0$).  It is the aim of this paper to extend the methods of Darbar and Maiti and study large values at the central point of the families  of cubic and quartic Dirichlet $L$-functions of prime moduli under GRH. \newline
 
    Our main result is as follows.  
\begin{theorem}
\label{main theorem 1}
 With the notation as above and the truth of GRH, we have, for sufficiently large $X$, 
\begin{equation} \label{cubicbound}
  \max_{\substack{X< p \le 2X \\\chi \bmod{p} \\ \chi^3 = \chi_0 }}\Big|L(\tfrac 12, \chi)\Big |  \geq \exp\left(\left(\frac {\sqrt{2}}{\pi} +o(1) \right)\frac {\log_{3} X}{\log_{2} X}\sqrt{\log X}\right),  \\
\end{equation}
and
\begin{equation} \label{quarticbound}
 \max_{\substack{X< p \le 2X \\\chi \bmod{p} \\ \chi^4 = \chi_0 }}\Big|L(\tfrac 12, \chi)\Big |   \geq \exp\left(\left(\frac {\sqrt{3/2}}{\pi} +o(1) \right)\frac {\log_{3} X}{\log_{2} X}\sqrt{\log X}\right).
\end{equation}
\end{theorem}

\section{Preliminaries}
\label{sec 2}

\subsection{Cubic and quartic Dirichlet characters}
\label{sec: characters}

Let $K$ be a number field and $\mathcal{O}_K, U_K$ its ring of integers and group of units, respectively.  We write $N(n)$ and $\mathrm{Tr}(n)$ for the norm and trace of any $n \in K$, respectively.   Let $\chi$ denote a Hecke character of $K$ and we say that $\chi$ is of trivial infinite type if its component at infinite places of $K$ is trivial.  We write $L(s,\chi)$ for the $L$-function associated to $\chi$ and $\zeta_{K}(s)$ the Dedekind zeta function of $K$.  We shall reserve $\varpi$ for a prime number in $\mathcal{O}_K$ throughout the paper, by which we mean that the ideal $(\varpi)$ generated by $\varpi$ is a prime ideal. \newline

  Now suppose that $K$ is a number field of class number one.  Let $\Lambda_K(n)$ denote the von Mangoldt function on $\mathcal{O}_K$ defined by
\begin{align*}
    \Lambda_K(n)=\begin{cases}
   \log N(\varpi) \qquad & (n)=(\varpi^k), \text{$\varpi$ prime}, k \geq 1, \\
     0 \qquad & \text{otherwise}.
    \end{cases}
\end{align*}  
Let $n \in \natn$ with $n \geq 2$ and $\mu_n(K):=\{ \zeta \in K^{\times}: \zeta^n=1 \}$ and suppose that $\mu_n(K)$ has $n$ elements.  We can easily compute the discriminant of $x^n-1$ as
\[ (-1)^{n(n-1)/2} R (x^n-1, nx^{n-1}) = (-1)^{n(n-1)/2} n^n P = (-1)^{n(n+1)/2 +1} n^n. \]
Here $R(f,g)$ is the resultant of the polynomials $f$ and $g$ and $P = (-1)^{n+1}$ is the product of the $n$-th roots of unity.  Hence the discriminant of $x^n-1$ is divisible only by the primes dividing $n$ in $\mathcal{O}_K$. It follows that for any prime $\varpi \in \mathcal{O}_K,  (\varpi, n)=1$, we have a bijection
\begin{align*}
    \zeta \mapsto \zeta \pmod \varpi : \mu_n(K) \rightarrow \mu_n(\mathcal{O}_K/\varpi):=\{ \zeta \in (\mathcal{O}_K/\varpi)^{\times}: \zeta^n=1 \}.
\end{align*}
Thus $(\mathcal{O}_K/\varpi)^{\times}$ contains a subgroup of order $n$ which implies that $n | N(\varpi)-1$.  For such $\varpi$,  we define the $n$-th power residue symbol $\leg{\cdot}{\varpi}_{n,K}$ in $K$ such that $\leg{a}{\varpi}_{n, K} \equiv a^{(N(\varpi)-1)/n} \pmod{\varpi}$ with $\leg{a}{\varpi}_{n,K} \in \mu_n(K)$ for any $a \in \mathcal{O}_K$, $(a, \varpi)=1$ (see \cite[Section 4.1]{Lemmermeyer}). When $\varpi | a$, we define $\leg{a}{\varpi}_{n, K} =0$.  Then these symbols can be extended to any composite $c$ with $(N(c), n)=1$ multiplicatively.  Further define $\leg {\cdot}{c}_{n, K}=1$ for $c \in U_K$. \newline

In the remainder of this section, we set $K=\mq(\omega)$ with $\omega=\exp(2\pi i/3)$ or $K=\mq(i)$ unless otherwise specified.  It is well-known that both fields have class number one and $\mathcal{O}_{K}=\mz[\omega]$, $\mz[i]$, respectively.  We use $\delta_K$ and $D_K$ for the different and discriminant of $K$, respectively.  In particular, we have $\delta_{\mq(\omega)}=\sqrt{-3}$, $\delta_{\mq(i)}=2i$, $D_{\mq(\omega)}=-3$, $D_{\mq(i)}=-4$.  We shall reserve the symbol $\leg {\cdot}{\cdot}_3$ for the cubic residue symbol $\leg {\cdot}{\cdot}_{3, \mq(\omega)}$ and $\leg {\cdot}{\cdot}_4$ for the quartic residue symbol $\leg {\cdot}{\cdot}_{4, \mq(i)}$ in this paper. \newline

Recall that every ideal in $\intz[\omega]$ co-prime to $3$ has a unique generator congruent to $1$ modulo $3$ (see \cite[Proposition 8.1.4]{BEW})
and every ideal in $\intz[i]$ co-prime to $2$ has a unique generator congruent to $1$ modulo $(1+i)^3$
(see the paragraph above Lemma 8.2.1 in \cite{BEW})). These generators are called primary. An element $n=a+b\omega$ in $\mz[\omega]$ is congruent to $1 \pmod{3}$ if and only if $a \equiv 1 \pmod{3}$, and $b \equiv
0 \pmod{3}$ (see the discussions before \cite[Proposition 9.3.5]{I&R}). \newline

If $\varpi$ is a prime such that $N(\varpi)$ is a rational prime, then restricting $\chi_{3, \varpi}:=\leg {\cdot}{\varpi}_3$ or $\chi_{4,\varpi}:=\leg {\cdot}{\varpi}_4$
on rational integers gives rise to cubic or quartic Dirichlet characters, and we shall say that these Dirichlet characters are induced by $\chi_{3, \varpi}$ or $\chi_{4,\varpi}$.  When no possible ambiguity exists, we write $\chi_{\varpi}$ for either $\chi_{3, \varpi}$ or $\chi_{4,\varpi}$. 
We have the following classification of primitive cubic and quartic Dirichlet characters of prime conductors from \cite[Lemma 2.4]{G&Zhao22-1}.
\begin{lemma}
\label{lemma:quarticclass}
 The primitive cubic Dirichlet characters of prime conductor $p$ co-prime to $3$ are induced by $\chi_{3, \varpi}$ for some prime $\varpi \in \mz[\omega]\setminus \mz$ with $N(\varpi) = p$.  The primitive quartic Dirichlet characters of prime conductor $p$ co-prime to $2$ such that their squares remain primitive are induced by $\chi_{4, \varpi}$ for some prime $\varpi \in \mz[i] \setminus \mz$ with $N(\varpi) = p$.
\end{lemma}

  In particular, the above lemma implies that cubic Dirichlet characters of prime conductor $p$ exist if and only if $p \equiv 1 \pmod 3$, in which case there are
two such characters induced by $\chi_{3, \varpi}$ or $\chi_{3, \overline{\varpi}}$, where $\varpi$ is a prime in $\mz[\omega]$ and $N(\varpi)=p$.  In a similar vein, quartic Dirichlet characters of prime conductor $p$ exist precisely for $p \equiv 1 \pmod 4$.  In that case, there are two such characters induced by $\chi_{4, \varpi}$ or $\chi_{4, \overline{\varpi}}$, where $\varpi$ is a prime in $\mz[i]$ with $N(\varpi)=p$. \newline

\subsection{The Gauss sums}
\label{sec2.4}

As usual, set $e(x)=\exp(2 \pi i x)$ and let $\chi$ be a Dirichlet character of modulus $n$. For any $r \in \mz$, the Gauss sum $\tau(r, \chi)$ is defined by
\begin{align*}
  \tau(r, \chi)=\sum_{x \bmod {n}}\chi(x)e(rx).
\end{align*}
We also write $\tau(\chi)$ for $\tau(1, \chi)$, for convenience. \newline

For any number field $K$, we further define $\widetilde{e}_K(n) =e(\mathrm{Tr}(n/ \delta_K))$ for any $n \in K$ and write $\widetilde{e}_{\omega}(z)$ and $\widetilde{e}_{i}(z)$ for $\widetilde{e}_{\mq(\omega)}(n)$ and $\widetilde{e}_{\mq(i)}(n)$, respectively. For any $n, r \in \mz[\omega], (n,3)=1$, set
\begin{align*}
 g_3(r,n) = \sum_{x \bmod{n}} \leg{x}{n}_{3} \widetilde{e}_{\omega}\leg{rx}{n}
\end{align*}
and, for any $n, r \in \mz[i], (n,2)=1$,
\begin{align*}
 g_4(r,n) = \sum_{x \bmod{n}} \leg{x}{n}_{3} \widetilde{e}_{i}\leg{rx}{n}.
\end{align*}
  We shall write $g_j(n)$ for $g_j(1,n)$ in what follows for $j=3, 4$. We note the well-known formula from \cite[p. 195]{P}):
\begin{align}
\label{2.1}
   |g_j(n)| =& \begin{cases}
    \sqrt{N(n)}, \qquad & \text{if $n$ is square-free}, \\
     0, \qquad & \text{otherwise}.
    \end{cases}
\end{align}

    Moreover, from \cite{Diac}, for $j=3$ and 4,
\begin{align}
\label{eq:gmult}
 g_j(rs,n)  = & \overline{\leg{s}{n}}_j g_j(r,n), \quad (s,n)=1.
\end{align}

If $\chi$ is induced by $\chi_{3,\varpi}$ for a primary $\varpi$, then (see  \cite[(2.7)]{G&Zhao2020-1})
\begin{align}
\label{taug3}
  \tau(\chi)  =&
    \displaystyle \overline{\leg {\sqrt{-3}}{\varpi}}_3 g_3(r, \varpi)=\overline{\leg {\omega(1-\omega)}{\varpi}}_3 g_3(r, \varpi).
\end{align}
For a quartic character $\chi_{4,\varpi}$ with a primary $\varpi$, it is shown in  \cite[(2.11)]{G&Zhao7} that
\begin{align}
\label{taug}
  \tau(\chi)=\begin{cases}
    \overline{ \leg {(-2i)^3}{n}}_4 g_4(n), \qquad & \text{if $\leg {-1}{n}_4=1$}, \\ \\
     i^{-1}\overline{\leg{(-2i)^3}{n}}_4 g_4(n), \qquad & \text{if $\leg {-1}{n}_4=-1$ }.
    \end{cases}
\end{align}  
    
   A key step in the proof of Theorem \ref{main theorem 1} is the following bound established in \cite[Theorem 4.4]{DG22} and \cite[Proposition 5.1]{G&Zhao24-02} on a Dirichlet polynomial formed from cubic and quartic Gauss sums, respectively.
\begin{prop}
\label{lemg3} With the notation as above, let $j=3$, $4$ and $\psi$ be any ray class character modulo $9$ (resp. $16$) when $j=3$ (resp. $j=4$) and let $r$ be any primary element in $\mathcal O_K$. Then we have for real $Z>1$ and $\varepsilon>0$, we have
\begin{align*}
\begin{split}
  H_Z(r;\psi) := \sum_{\substack {c \odd \\ (c, r)=1 \\ N(c) \leq Z}} \frac {\Lambda_K(c)\psi(c) g_{K,j}(r, c) }{\sqrt{N(c)}}  \ll &
\begin{cases}
   \min (Z^{1+\varepsilon},  N(r)^{1/10+\varepsilon}Z^{4/5+\varepsilon}), & j=3, \\ \\
   \min (Z^{1+\varepsilon},  N(r)^{1/16+\varepsilon}Z^{7/8+\varepsilon}), & j=4.
\end{cases}
\end{split}
\end{align*}
\end{prop}

\subsection{The approximate functional equation}
\label{sec: afe}

   Let $\chi$ be any primitive Dirichlet character modulo $q$ and $\af=0$ or $1$ be given by $\chi(-1)=(-1)^{\af}$. We define
\begin{align*}
  \Lambda(s, \chi)= \left( \frac {\pi}{q} \right)^{-(s+\af)/2}\Gamma \left( \frac {s+\af}{2} \right)L(s, \chi).
\end{align*}
   Then $\Lambda(s, \chi)$ extends to an entire function on $\mc$ when $\chi \neq \chi_0$ and satisfies the functional equation (see \cite[Theorem 4.15]{iwakow})
\begin{align*}
  \Lambda(1-s, \overline \chi)=\frac {i^{\af}q^{1/2}}{\tau(\chi)} \Lambda(s, \chi).
\end{align*}

Let $G(s)$ be any even function which is holomorphic and bounded in the strip $-4<\Re(s)<4$ satisfying $G(0)=1$. From \cite[Theorem 5.3]{iwakow}, we have the following approximate functional equation for Dirichlet $L$-functions.
\begin{prop} \label{prop:AFE}
Suppose $\chi$ be a primitive Dirichlet character modulo $q$. Let $A$ and $B$ be positive real numbers with $AB = q$.  Then we have
\begin{align*}
L \left( \frac{1}{2} , \chi \right) = \sum_{m=1}^{\infty} \frac{\chi(m)}{m^{1/2 }} V_{\af}\left(\frac{m}{A}\right)
+ \epsilon(\chi) \sum_{m=1}^{\infty} \frac{\overline{\chi}(m)}{m^{1/2}} V_{\af}\left(\frac{m}{B}\right),
\end{align*}
  where
\begin{align*}
\epsilon(\chi) = i^{-\af} q^{-1/2} \tau(\chi), \quad V_{\af}(x) = \frac{1}{2\pi i} \int\limits_{(2)} \frac{G(s)}{s} \gamma_{\af}(s) x^{-s} \dif s,  \quad \gamma_{\af}(s) = \pi^{-s/2} \frac{\Gamma\left(\tfrac{1/2 + \af+ s}{2}\right)}{\Gamma\left(\tfrac{1/2 + \af}{2}\right)}.
\end{align*}
\end{prop}
If $\chi$ is a cubic character, then $\chi(-1)=\chi((-1)^3)=1$.  It follows that $\af=0$ in this case. We now set $G(s)=1$ and for brevity write $V$ for $V_0$.  Hence if $\chi$ is a cubic primitive Dirichlet character, then
\begin{align}
\label{approxfunc}
L \left( \tfrac{1}{2} , \chi \right) = \sum_{m=1}^{\infty} \frac{\chi(m)}{m^{1/2 }} V\left(\frac{m}{A}\right) + \epsilon(\chi) \sum_{m=1}^{\infty} \frac{\overline{\chi}(m)}{m^{1/2}} V\left(\frac{m}{B}\right). 
\end{align}    
For the quartic case, again set $G(s)=1$ and relabel $V_{\af}$ by $V_{\pm}$ according to $\af=0$ or $1$, arriving at
\begin{align}
\label{approxfuncquartic}
L \left( \tfrac{1}{2} , \chi \right) =
\begin{cases} 
\displaystyle \sum_{m=1}^{\infty} \frac{\chi(m)}{m^{1/2 }} V_{+}\left(\frac{m}{A}\right)
+ \epsilon(\chi) \sum_{m=1}^{\infty} \frac{\overline{\chi}(m)}{m^{1/2}} V_{+}\left(\frac{m}{B}\right), \quad \chi(-1)=1, \\ \\
\displaystyle  \sum_{m=1}^{\infty} \frac{\chi(m)}{m^{1/2 }} V_{-}\left(\frac{m}{A}\right)
+ \epsilon(\chi) \sum_{m=1}^{\infty} \frac{\overline{\chi}(m)}{m^{1/2}} V_{-}\left(\frac{m}{B}\right), \quad \chi(-1)=-1. 
\end{cases}
\end{align}    

    Note that for any real number $x>0$,
\begin{align*}
 V(x) = \frac{1}{2 \pi i} \int\limits\limits_{(2)}  \left(\frac{1}{\pi}\right)^{s/2}
 \frac {\Gamma(s/2+1/4)}{\Gamma(1/4)}  x^{-s} \frac {\dif s}{s}.
\end{align*}

  Similar to what is shown in \cite[Lemma 2.1]{sound1}, the function $V(x)$ is real-valued, smooth on $(0, +\infty)$ and for any $\varepsilon>0$, 
\begin{equation} 
\label{2.07}
      V\left (x \right) = 1+O(x^{1/2-\varepsilon}) \;\; \mbox{for} \;\; 0<x <1   \quad \mbox{and} \quad V^{(j)}\left (x \right) =O(e^{-x}) \;\; \mbox{for}
      \;\; x >0, \; j \geq 0.
\end{equation}
It is shown in \cite{Gao2026-6} that $V(x)$ is positive for $x \geq 0$. The above properties are valid for $V_{\pm}(x)$ as well. \newline
  
   When $\chi=\chi_{\varpi}$ for some primary $\varpi \in \mz[\omega] \setminus \mz$ with $N(\varpi) = p$, then we deduce from \eqref{2.1}, \eqref{taug3} and \eqref{approxfunc} that
\begin{align}
\label{approxfuncubic}
L \left( \tfrac{1}{2} , \chi_{\varpi} \right) = \sum_{m=1}^{\infty} \frac{\chi_{\varpi}(m)}{m^{1/2 }} V\left(\frac{m}{A}\right)
+  N(\varpi)^{-1/2}  \sum_{m=1}^{\infty} \frac{\overline \chi_{\varpi}(\omega(1-\omega)m) g_3(\varpi)}{m^{1/2}} V\left(\frac{m}{B}\right). 
\end{align}   

If $\chi=\chi_{\varpi}$ for some primary prime $\varpi \in \mz[i] \setminus \mz$ with $N(\varpi) = p$, then \eqref{2.1}, \eqref{taug} and \eqref{approxfuncquartic} similarly yield
\begin{align}
\label{approxfuncquarticprime}
L \left( \tfrac{1}{2} , \chi_{\varpi} \right) = 
\begin{cases} 
\displaystyle \sum_{m=1}^{\infty} \frac{\chi_{\varpi}(m)}{m^{1/2}} V_{+}\left(\frac{m}{A}\right)
+  N(\varpi)^{-1/2}  \sum_{m=1}^{\infty} \frac{\overline \chi_{\varpi}((-2i)^3m) g_4(\varpi)}{m^{1/2}} V_{+}\left(\frac{m}{B}\right), \quad \chi_{\varpi}(-1)=1, \\ \\
\displaystyle\sum_{m=1}^{\infty} \frac{\chi_{\varpi}(m)}{m^{1/2}} V_{-}\left(\frac{m}{A}\right)
-  N(\varpi)^{-1/2}  \sum_{m=1}^{\infty} \frac{\overline \chi_{\varpi}((-2i)^3m) g_4(\varpi)}{m^{1/2}} V_{-}\left(\frac{m}{B}\right), \quad \chi_{\varpi}(-1)=-1.
\end{cases}
\end{align}

\subsection{Smoothed character sums}
\label{smoothsum}

We reserve $\psi_{m}$ for the Hecke characters in $K$ with $\psi_{m}((n)) = \leg{m}{n}_3$ for $n \in \mq(\omega)$ coprime to $3$
   or $\psi_{m}((n)) = \leg{m}{n}_4$ for $n \in \mq(i)$ coprime to $2$.
 It is shown in \cite[Section 2.1]{B&Y} and \cite[Section 2.1]{G&Zhao7} that $\psi_m$ is either a cubic  Hecke character of trivial infinite type
  modulo $9m$ or a   quartic Hecke character of trivial infinite type modulo $16m$. We define $\delta_{n=\text{cubic}}$ to be $1$ or $0$ depending on whether $n$ equals a cube or not, and $\delta_{n=\text{quartic}}$ similarly.  Analogous to \cite[Proposition 1]{Radziwill&Sound}, we need to bound smoothed sums of cubic and quartic characters in this paper. In the sequel, let $\Phi$ be a smooth, non-negative function compactly supported on $[1,2]$ and $\Phi(x) =1$ for $x\in [5/4,7/4]$.  Denote the Mellin transform of $\Phi(x)$ by ${\widehat \Phi}(s)$ so that for any complex number $s$,
\begin{equation*}
{\widehat \Phi}(s) = \int\limits_{0}^{\infty} \Phi(x)x^{s}\frac {\dif x}{x}.
\end{equation*}
   We have the following result on the smoothed cubic and quartic Dirichlet character sums.
\begin{lemma}
\label{PropDirpoly}  With the notation as above and assuming the truth of GRH, we have, for large $X$ and any positive integer $c \in \mz$, 
\begin{equation} \label{cubiccharsum}
\sum_{\substack{(p,3)=1}} \ \sumstar_{\substack{\chi \shortmod{p} \\ \chi^3 = \chi_0}} (\log p)\chi(c) \Phi\Big(\frac{p}{X}\Big)=
\displaystyle \delta_{c=\text{cubic}} {\widehat \Phi}(1) X  + O( X^{1/2+\varepsilon}\log (c+2) ),
\end{equation}
and
\begin{equation} \label{quarticcharsum}
\sum_{(p,2)=1} \ \sumstar_{\substack{\chi \shortmod{p} \\ \chi^4 = \chi_0}} \chi(c) \Phi\Big(\frac{p}{X}\Big)=
\displaystyle \delta_{c=\text{quartic}} {\widehat \Phi}(1) X + O( X^{1/2+\varepsilon}\log (c+2) ).
\end{equation}
Here and after the asterisk on the sum over $\chi$ in \eqref{cubiccharsum} restricts the sum to primitive characters and $\chi_0$ denotes the principal character, while the asterisk on the sum over $\chi$ in \eqref{quarticcharsum} limits the sum to primitive characters $\chi$ such that $\chi^2$ remains primitive.
\end{lemma}
\begin{proof}
  As the proofs for both cases are similar, we shall only prove \eqref{cubiccharsum}. Thus, in the remaining of the proof, we shall work exclusively over the number field $K=\mq(\omega)$ with all the related notations (e.g. primary) being referred to those related to this $K$.  Let $CS$ denote the left-hand side of \eqref{cubiccharsum}.  Applying Lemma \ref{lemma:quarticclass} gives
\begin{align}
\label{CS}
\begin{split}
 CS=& \sum_{\substack{ \varpi \odd \\ \varpi \notin \mz }}(\log N(\varpi))\chi_{\varpi}(c)\Phi\left(\frac{N(\varpi)}{X}\right)\\
=& \sum_{\substack{ \varpi \odd  }}(\log N(\varpi))\chi_{\varpi}(c)\Phi\left(\frac{N(\varpi)}{X}\right) +O\Big( \Big| \sum_{\substack{ p \in \mz \\ p \neq 3}}(\log N(p))\chi_{p}(c)\Phi\Big(\frac{N(p)}{X}\Big) \Big| \Big).  
\end{split}
\end{align}
  
For primes $p \in \mz$, $N(p)=p^2$. It follows that for any $\varepsilon>0$, 
\begin{align}
\label{CSerror}
\begin{split}
 \sum_{\substack{ p \in \mz \\ p \neq 3}}(\log N(p))\chi_{p}(c)\Phi\left(\frac{N(p)}{X}\right) \ll \sum_{\substack{ p \in \mz}}(\log p)\Phi\left(\frac{p^2}{X}\right)  \ll X^{1/2+\varepsilon}. 
\end{split}
\end{align}    
Thus, from \eqref{CS} and \eqref{CSerror},
\begin{align}
\label{CS1}
\begin{split}
 CS
=& \sum_{\substack{ \varpi \odd  }}(\log N(\varpi))\chi_{\varpi}(c)\Phi\left(\frac{N(\varpi)}{X}\right) +O\big ( X^{1/2+\varepsilon}\big ).  
\end{split}
\end{align}

 As $\Phi$ has compact support, we get
\begin{align} \label{CS2}
 CS= \sum_{\substack{ \varpi \odd  }}(\log N(\varpi))\chi_{\varpi}(c)\Phi\left(\frac{N(\varpi)}{X}\right)+O\big ( X^{1/2+\varepsilon}\big )
 = \sum_{n \odd } \chi_{n}(c) \Lambda_{\omega}(n) \Phi \left( \frac {N(n)}X \right)
 +O \left( X^{1/2+\varepsilon}\right), 
\end{align}
Via standard arguments, replacing $\log N(\varpi)$ by $\Lambda_{\omega}(n)$ incur an error that can be absorbed into the $O$-term above. \newline
 
Now apply Mellin inversion renders
\begin{align} \label{CS4}
\sum_{n \odd }  \chi_{n}(c) \Lambda_{\mq(\omega)}(n) \Phi \left( \frac {N(n)}X \right) = -\frac {1}{2\pi i}\int\limits_{(2)} \frac {L'(s, \psi_c)}{L(s, \psi_c)} \widehat{\Phi}(s)X^s \dif s.
\end{align}

  Note that integration by parts shows that $\widehat{\Phi}(s)$ satisfies the bound
\begin{align} \label{boundsforphi}
  \widehat{\Phi}(s) \ll \min (1, |s|^{-1}(1+|s|)^{-E}), \; \mbox{for all} \; \Re(s) > 0 \; \mbox{and integers} E>0.
\end{align}

  As discussed in Section \ref{sec: characters}, the Hecke character $\psi_c$ is a cubic Hecke character of trivial infinite type
  modulo $9c$. Suppose that it is induced by a primitive Hecke character $\widetilde \psi_c$ of trivial infinite type modulo $c'$, then $c'|9c$. Moreover, 
\begin{align*}
  L(s, \psi_c)=L(s, \widetilde\psi_c)\prod_{\substack{ (\varpi) \\ \varpi \mid 3c}}(1-\widetilde\psi_c(\varpi)N(\varpi)^{-s}).
\end{align*}   
  It follows that
\begin{align}
\label{logLder}
  -\frac {L'(s, \psi_c)}{L(s, \psi_c)}=-\frac {L'(s, \widetilde \psi_c)}{L(s, \widetilde \psi_c)}-\sum_{\substack{ (\varpi) \\ \varpi \mid 3c}}\frac {\widetilde\psi_c(\varpi)(\log N(\varpi))N(\varpi)^{-s}}{1-\widetilde\psi_c(\varpi)N(\varpi)^{-s}}.
\end{align}     
For $\Re(s) \geq 1/2$, we have
\begin{align}
\label{sumprimebound}
 \sum_{\substack{ (\varpi) \\ \varpi \mid 3c}}\frac {\widetilde\psi_c(\varpi)(\log N(\varpi))N(\varpi)^{-s}}{1-\widetilde\psi_c(\varpi)N(\varpi)^{-s}} \ll \log (N(c)+2).
\end{align}    
  Also, it follows from \cite[Theorem 5.17]{iwakow}) that under GRH, we have for $\Re(s) \geq 1/2+\varepsilon, |s-1|>1/100$,
\begin{align}
\label{Lderbound}
  \frac {L'(s, \widetilde \psi_c)}{L(s, \widetilde \psi_c)}  \ll \log\big ((N(c)+2)(1+|s|)\big).
\end{align}
  
  We now evaluate the integral in \eqref{CS4} by shifting the line of integration to $\Re(s)=1/2+\varepsilon$.  We encounter a pole at $s=1$ with residue
  $\widehat{\Phi}(1)X$ only if $c$ is a perfect cube.  The integration on the line $\Re(s)=1/2+\varepsilon$ can be estimated as $O(X^{1/2+\varepsilon}\log  (N(c)+2))$ by using \eqref{boundsforphi} and \eqref{logLder}--\eqref{Lderbound}.  As $c \in \mz$, we have $N(c)=c^2$. This, together with \eqref{CS4}, leads to \eqref{cubiccharsum} and completes the proof.
\end{proof}

\section{Proof of the Theorem \ref{main theorem 1}}
\subsection{Initial Treatments}

  Let $[x]$ denote the largest integer not exceeding $x$ for any real $x$ and $N \ll X$ be a large number to be fixed later.  Suppose $a\in (1, \infty)$ and $\delta\in (0,1)$ are fixed with $1<a< 1/\delta$. We define for $k=1, \ldots, [(\log_{2}N)^{\delta}]$ the set $\mathcal{P}_{k}$ to consist all primes $p$ such that
 \[ e^{k}\log N \log_{2} N < p \le e^{k+1}\log N \log_{2} N. \]
Let $\mathcal{P} = \cup_k \mathcal{P}_{k}$.  Thus $\mathcal{P}$ consists of all primes $p$ satisfying
\[ e\log N \log_{2} N< p \le e^{(\log_{2} N)^{\delta}} \log N \log_{2} N. \]
Denote by $ \mathcal{R}_{k}$ the set of positive integers that have at least $\frac{a\log N}{k^{2}\log_{3}N}$ prime divisors in $\mathcal{P}_{k}$. \newline
 
Define a multiplicative function $\psi$ supported on the set of square-free integers in the following manner.  Let $\psi(p)=0$ for $p \notin \mathcal{P}$ and for any $p \in\mathcal{P}$,
\begin{align}
\label{psidef}
\psi(p)=\sqrt{\frac{\log N \log_{2}N }{\log_{3}N}}p^{-1/2}\left(\log p-\log(\log N \log_{2}N)\right)^{-1}.
\end{align}
 Further set
\begin{align}
\label{Rdef}
\mathcal{R}:=\rm{supp}(\psi)\setminus\displaystyle \bigcup_{k=1}^{[(\log_{2}N)^{\delta}]}\mathcal{R}_{k}.
\end{align} 
  The resonator for $L(\frac 1 2, \chi)$ for any Dirichlet character $\chi$ is then defined to be the Dirichlet polynomial 
\[ R_{\chi}:= \sum_{m\in \mathcal{R}}\psi(m)\chi(m). \] 
 
We quote the following result from \cite[Lemma 3]{DM25} on the size of $\mathcal{R}$, which clearly depends on $N$.  
\begin{lemma}\label{Le1} 
With the notation as above and $1<a<1/\delta$, we have $|\mathcal{R}|\leq N$ for $N$ large enough.
\end{lemma}

Now define 
\begin{align*}
\mathcal{A}_{N}:=\frac{1}{\displaystyle \sum_{m \in \mn } \psi(m)^{2} }\sum_{n\in \mn }\frac{\psi{(n)}}{\sqrt{n}}\sum_{\substack{l|n }}\psi(l)\sqrt{l}.
\end{align*}

   We further cite three results pertaining to $\mathcal{A}_{N}$ from \cite[ Lemmas 1--3]{BS18}.
\begin{lemma}\label{Le2}
	With the notation as above, we have as $N\to\infty$,
	\[\mathcal{A}_{N}\ge \exp\left((\delta + o(1))\sqrt{\frac{\log N\log_{3}N}{\log_{2}N}}\right).
	\]	
\end{lemma}

\begin{lemma}\label{Le3}
With the notation as above, we have as $N\to\infty$,
	\begin{align*}
	\frac{1}{\displaystyle\sum_{m\in \mn} \psi(m)^{2}}\sum_{\substack{n\in \mn \\  n\notin\mathcal{R}}}\frac{\psi{(n)}}{\sqrt{n}}\sum_{\substack{l|n }}\psi(l)\sqrt{l}=o(\mathcal{A}_{N}).
	\end{align*}
\end{lemma}  

\begin{lemma}\label{Le5}
With the notation as above, for any $\varepsilon>0$, we have as $N\to\infty$, 
\[ \frac{1}{\displaystyle \sum_{m \in \mn} \psi(m)^{2}}\sum_{\substack{n\in \mathcal{R} }}\frac{\psi{(n)}}{\sqrt{n}}\sum_{\substack{l|n\\ l \le n/N^{\varepsilon}}}\psi(l)\sqrt{l}=o(\mathcal{A}_{N}), \]
where the implicit constant only depends on $\varepsilon$.
\end{lemma}

\subsection{Proof of \eqref{cubicbound}}
Let $\Phi$ be the function described in Section \ref{smoothsum}.  We set
\begin{align*} 
 \mathcal{S}_{1}:=\sum_{\substack{(p,3)=1}} \ \sumstar_{\substack{\chi \bmod{p} \\ \chi^3 = \chi_0}}  (\log p)L(\tfrac{1}{2},\chi) |R_{\chi}|^2 \Phi \Big( \frac {p}{X} \Big) \quad \mbox{and} \quad \mathcal{S}_{2}:=\sum_{\substack{(p,3)=1}} \ \sumstar_{\substack{\chi \bmod{p} \\ \chi^3 = \chi_0}}(\log p)|R_{\chi}|^2\Phi \Big(\frac {p}{X} \Big).
\end{align*}
  As $|R^2_{\chi}| \geq 0$, we then observe that
\begin{align}  
\label{maxlower}
 \max_{\substack{X< p \le 2X \\\chi \bmod{p} \\ \chi^3 = \chi_0 }}\Big|L(\tfrac 12, \chi)\Big |\ge \frac{|\mathcal{S}_1|}{\mathcal{S}_2}. 
\end{align}
Thus It remains to estimate $|\mathcal{S}_{1}|$ from below and $\mathcal{S}_{1}$ from above.  Applying \eqref{approxfunc} and \eqref{approxfuncubic} leads to 
\begin{align}
\label{S1eval}
\begin{split}
	\mathcal{S}_{1}=& S_{1,1}+S_{1,2},   
\end{split}  
\end{align}
   where for some constant $A=A(X)>0$, 
\begin{equation} \label{S11def}
	\mathcal{S}_{1,1}= \sum_{m, n\in \mathcal{R}} \psi(m) \psi(n)\sum_{l\geq 1 } \frac{1}{\sqrt{l}}\sum_{\substack{(p,3)=1}} \ \sumstar_{\substack{\chi \shortmod{p} \\ \chi^3 = \chi_0}}  (\log p)\chi(lm)\overline{\chi}(n)V\left(\frac{l}{A\sqrt{p}} \right)\Phi \Big(\frac {p}{X}\Big), 
	\end{equation}
	and
	\begin{equation} \label{S12def}
	\mathcal{S}_{1,2}= \sum_{m, n\in \mathcal{R}} \psi(m) \psi(n)\sum_{l\geq 1 } \frac{1}{\sqrt{l}} \sum_{\substack{ \varpi \odd \\ \varpi \notin \mz }}\frac {(\log N(\varpi))\chi_{\varpi}(m)\overline\chi_{\varpi}(\omega(1-\omega)ln)g_3(\varpi)}{\sqrt{N(\varpi)}}V\left(\frac{lA}{\sqrt{N(\varpi)}} \right)\Phi\left(\frac{N(\varpi)}{X}\right) .   
\end{equation}   

We treat $\mathcal{S}_{1,1}$ first.  As $\chi$ is cubic, $\overline \chi(n)=\chi(n^2)$.  Applying Lemma \ref{PropDirpoly}  and partial summation, the inner double sum over $p$ and $\chi$ in $\mathcal{S}_{1,1}$ becomes
\begin{align*}
\begin{split}
\sum_{\substack{(p,3)=1}} \ \sumstar_{\substack{\chi \shortmod{p} \\ \chi^3 = \chi_0}} & (\log p)\chi(lmn^2)V\left(\frac{l}{A\sqrt{p}} \right)\Phi \Big(\frac {p}{X}\Big) \\
=& \int\limits^{2X}_X V\left(\frac{l}{A\sqrt{t}} \right) \dif \left( t\delta_{lmn^2=\text{cubic}}{\widehat \Phi}(1)+ O\left(t^{1/2+\varepsilon}\log
  (lmn^2+2) \right)  \right), \\
=& {\widehat \Phi}(1)X \delta_{lmn^2=\text{cubic}}\int\limits^{2}_1 V\left(\frac{l}{A\sqrt{Xt}} \right) \dif t \\
& \hspace*{1cm} +V\left(\frac{lA}{\sqrt{t}} \right)O\left(t^{1/2+\varepsilon}\log
  (lmn+2) \right)\Bigg|^{2X}_X -\int\limits^{2X}_X O\left(t^{1/2+\varepsilon}\log
 (lmn+2) \right)V'\left(\frac{l}{A\sqrt{t}}\right ) \frac {l}{2At^{3/2}} \dif t \\
=: & {\widehat \Phi}(1)X  \delta_{lmn^2=\text{cubic}}\int\limits^{2}_1 V\left(\frac{l}{A\sqrt{Xt}} \right) \dif t+R. 
\end{split}
\end{align*}

Hence,
\begin{align*}
    \mathcal{S}_{1,1}={\widehat \Phi}(1)X
     \sum_{m, n \in \mathcal{R}}&\psi(m)\psi(n)\sum_{\substack{lmn^2=\text{cubic}}} \frac{1}{\sqrt{l}}\int\limits_{1}^{2}V\left(\frac{l}{A\sqrt{Xt}} \right) dt+O\Bigg(\sum_{m, n\in \mathcal{R}} \psi(m) \psi(n)\sum_{l\geq 1 } \frac{R}{\sqrt{l}} \Bigg).
\end{align*}
Mindful of the rapid decay of $V$ and $V'$ given in \eqref{2.07}, we may restrict the sum over $l$ to be $l \leq (A\sqrt{X})^{1+\varepsilon}$ for any $\varepsilon>0$ in the $O$-term above.  It follows that
\begin{align*}	
    R \ll  X^{1/2+\varepsilon}.
\end{align*}    
   We deduce from the above, and Cauchy's inequality, that
\begin{align}	
\label{S1exp}
\begin{split}
    \mathcal{S}_{1,1}=& {\widehat \Phi}(1)X\sum_{m, n \in \mathcal{R}}\psi(m)\psi(n)\sum_{\substack{ lmn^2=\text{cubic}}} \frac{1}{\sqrt{l}}\int\limits_{1}^{2}V\left(\frac{l}{A\sqrt{Xt}} \right) \dif t  +O\Bigg(X^{1/2 +\varepsilon}\sum_{\substack{l \leq A^{1+\varepsilon}X^{1/2+\varepsilon}}}\frac{1}{\sqrt{l}}  \left(\sum_{m\in \mathcal{R}}\psi(m)\right)^2\Bigg) \\
    =& {\widehat \Phi}(1)X\sum_{m, n \in \mathcal{R}}\psi(m)\psi(n)\sum_{\substack{ lmn^2=\text{cubic}}} \frac{1}{\sqrt{l}}\int\limits_{1}^{2}V\left(\frac{l}{A\sqrt{Xt}} \right) \dif t+O\Bigg(X^{3/4 +\varepsilon}A^{1/2+\varepsilon}|\mathcal{R}|\sum_{m\in \mathcal{R}}\psi^2(m)\Bigg).
\end{split}
\end{align}
 
On to $\mathcal{S}_{1,2}$, we recast the inner sum over $\varpi$ as
\begin{align}
\label{S12def1}
\begin{split}
\sum_{\substack{ \varpi \odd}} & \frac {(\log N(\varpi))\chi_{\varpi}(m)\overline\chi_{\varpi}(\omega(1-\omega)ln)g_3(\varpi)}{\sqrt{N(\varpi)}}V\left(\frac{lA}{\sqrt{N(\varpi)}} \right)\Phi\left(\frac{N(\varpi)}{X}\right) \\
&+ O \left ( \Bigg| \sum_{\substack{ p \in \mz}}\frac {(\log N(p))\chi_{p}(m)\overline\chi_{p}(\omega(1-\omega)ln)g_3(p)}{\sqrt{N(p)}}V\left(\frac{lA}{\sqrt{N(p)}} \right)\Phi\left(\frac{N(p)}{X}\right) \Bigg| \right).   
\end{split}  
\end{align}    
Note that $N(p)=p^2$ for rational prime $p$.  From \eqref{2.1} and the compact support of $\Phi$, we get
\begin{align}
\label{S12errorest1}
\begin{split}
\sum_{\substack{ p \in \mz}}\frac {(\log N(p))\chi_{p}(m)\overline\chi_{p}(\omega(1-\omega)ln)g_3(p)}{\sqrt{N(p)}}V\left(\frac{lA}{\sqrt{N(p)}} \right)\Phi\left(\frac{N(p)}{X}\right) \ll \sum_{\substack{ p \ll \sqrt{X}}}\log p \ll X^{1/2}. 
\end{split}  
\end{align}   
From this, \eqref{S12def1} becomes
\begin{align}
\label{S12simplified}
\begin{split}
\sum_{\substack{ c \odd}} & \frac { \Lambda_{\mq(\omega)}(c)\chi_{c}(m)\overline\chi_{c}(\omega(1-\omega)ln)g_3(c)}{\sqrt{N(c)}}V\left(\frac{lA}{\sqrt{N(c)}} \right)\Phi\left(\frac{N(c)}{X}\right) \\
&+ O \Bigg( \Bigg| \sum_{\substack{ \varpi^j \\ \varpi \odd, j \geq 2}}\frac { \Lambda_{\mq(\omega)}(\varpi^j)\chi_{\varpi^j}(m)\overline\chi_{\varpi^j}(\omega(1-\omega)ln)g_3(\varpi^j)}{\sqrt{N(\varpi^j)}}V\left(\frac{lA}{\sqrt{N(\varpi^j)}} \right)\Phi\left(\frac{N(\varpi^j)}{X}\right) \Bigg| + X^{1/2} \Bigg). 
\end{split}  
\end{align}  
Arguing as in \eqref{S12errorest1}, we get, for any $\varepsilon >0$, 
\begin{align*}
\begin{split}
\sum_{\substack{ \varpi^j \\ \varpi \odd, j \geq 2}}\frac { \Lambda_{\mq(\omega)}(\varpi^j)\chi_{\varpi^j}(m)\overline\chi_{\varpi^j}(\omega(1-\omega)ln)g_3(\varpi^j)}{\sqrt{N(\varpi^j)}}V\left(\frac{lA}{\sqrt{N(\varpi^j)}} \right)\Phi\left(\frac{N(\varpi^j)}{X}\right) \ll X^{1/2+\varepsilon}.
\end{split}  
\end{align*}   
  We deduce from the above and \eqref{S12simplified} that \eqref{S12def1} is
\begin{align}
\label{S12simplified1}
\begin{split}
 \sum_{\substack{ c \odd}} & \frac { \Lambda_{\mq(\omega)}(c)\chi_{c}(m)\overline\chi_{c}(\omega(1-\omega)ln)g_3(c)}{\sqrt{N(c)}}V\left(\frac{lA}{\sqrt{N(c)}} \right)\Phi\left(\frac{N(c)}{X}\right) + O \left (X^{1/2+\varepsilon} \right ) \\
=& \sum_{\substack{ c \odd \\ (c, lmn)=1}}\frac { \Lambda_{\mq(\omega)}(c)\overline\chi_{c}(\omega(1-\omega))g_3(lnm^2, c)}{\sqrt{N(c)}}V\left(\frac{lA}{\sqrt{N(c)}} \right)\Phi\left(\frac{N(c)}{X}\right) + O \left (X^{1/2+\varepsilon} \right ),
\end{split}  
\end{align}
where we used the observation that $\chi(m)=\overline \chi(m^2)$ as $\chi$ is cubic and \eqref{eq:gmult}. \newline

  Moreover, due to the rapid decay of $V$ given in \eqref{2.07}, we may restrict the sum over $l$ in $S_{1,2}$ to $l \leq (\sqrt{X}/A)^{1+\varepsilon}$. It follows from this, \eqref{S12def}, \eqref{S12simplified1} and Cauchy's inequality that 
\begin{align}
\label{S12}
\begin{split}
	\mathcal{S}_{1,2}=&\sum_{m, n\in \mathcal{R}} \psi(m) \psi(n)\sum_{l\geq 1 } \frac{1}{\sqrt{l}}\sum_{\substack{ c \odd\\ (c, lmn)=1}}\frac { \Lambda_{\mq(\omega)}(c)\overline\chi_{c}(\omega(1-\omega))g_3(lnm^2, c)}{\sqrt{N(c)}}V\left(\frac{lA}{\sqrt{N(c)}} \right)\Phi\left(\frac{N(c)}{X}\right) \\
& \hspace*{2cm} +O\left(X^{1/2+\varepsilon} \sum_{m, n\in \mathcal{R}} \psi(m) \psi(n)\sum_{l \leq (\sqrt{X}/A)^{1+\varepsilon} } \frac{1}{\sqrt{l}}\right) \\
=&\sum_{m, n\in \mathcal{R}} \psi(m) \psi(n)\sum_{l\geq 1 } \frac{1}{\sqrt{l}}\sum_{\substack{ c \odd\\ (c, lmn)=1}}\frac { \Lambda_{\mq(\omega)}(c)\overline\chi_{c}(\omega(1-\omega))g_3(lnm^2, c)}{\sqrt{N(c)}}V\left(\frac{lA}{\sqrt{N(c)}} \right)\Phi\left(\frac{N(c)}{X}\right) \\
& \hspace*{2cm} +O\left(X^{3/4+\varepsilon}A^{-1/2+\varepsilon}|\mathcal{R}| \sum_{m\in \mathcal{R}} \psi^2(m)\right).
\end{split}
\end{align}
 
   Note that $\overline\chi_{c}(\omega(1-\omega))=\chi_{c}((\omega(1-\omega))^2)$ and the Hecke character $\psi_{(\omega(1-\omega))^2}$ is a ray class character modulo $9$. This allows us to apply Proposition \ref{lemg3} and get
\begin{align*}
\begin{split}
	\sum_{\substack{ c \odd\\ (c, lmn)=1}} & \frac { \Lambda_{\mq(\omega)}(c)\overline\chi_{c}(\omega(1-\omega))g_3(lnm^2, c)}{\sqrt{N(c)}}V\left(\frac{lA}{\sqrt{N(c)}} \right)\Phi\left(\frac{N(c)}{X}\right) \ll  X^{\varepsilon}\int\limits^{2X}_XV\left(\frac{lA}{\sqrt{t}} \right)\Phi\left(\frac{t}{X}\right)\dif (O(t^{4/5}N(lnm^2)^{1/10})) \\
\ll & X^{\varepsilon}O(t^{4/5}N(lnm^2)^{1/10})V\left(\frac{lA}{\sqrt{t}} \right)\Phi\left(\frac{t}{X}\right)\Bigg|^{2X}_X \\
& \hspace*{2cm} -\int\limits^{2X}_XO(t^{4/5}N(lnm^2)^{1/10}) \Bigg(V'\left(\frac{lA}{\sqrt{t}} \right)\Phi\left(\frac{t}{X}\right) \left(\frac {-lA}{2t^{3/2}}\right)+V\left(\frac{lA}{\sqrt{t}} \right)\Phi'\left(\frac{t}{X}\right)\frac 1{X} \Bigg) \dif t. 
\end{split}  
\end{align*}    
  By the rapid decay of $V'(x)$ given in \eqref{2.07}, we may assume that $l \leq (\sqrt{t}/A)^{1+\varepsilon}$ for each fixed $t$ in the integral above. It follows from this and the above that
\begin{align*}
\begin{split}
\sum_{\substack{ c \odd\\ (c, lmn)=1}}\frac { \Lambda_{\mq(\omega)}(c)\overline\chi_{c}(\omega(1-\omega))g_3(lnm^2, c)}{\sqrt{N(c)}}V\left(\frac{lA}{\sqrt{N(c)}} \right)\Phi\left(\frac{N(c)}{X}\right) \ll  X^{4/5+\varepsilon}(lnm^2)^{1/5+\varepsilon}. 
\end{split}  
\end{align*}   
Inserting the above into \eqref{S12} and restricting the sum over $l$ to $l \leq (\sqrt{X}/A)^{1+\varepsilon}$, we arrive at
\begin{align}
\label{S12bound}
\begin{split}
	\mathcal{S}_{1,2} \ll & \sum_{m, n\in \mathcal{R}} \psi(m) \psi(n)\sum_{l\leq (\sqrt{X}/A)^{1+\varepsilon}} \frac{X^{4/5+\varepsilon}}{\sqrt{l}}(lnm^2)^{1/5+\varepsilon}+O\left(X^{3/4+\varepsilon}A^{-1/2+\varepsilon}|\mathcal{R}| \sum_{m\in \mathcal{R}} \psi^2(m)\right) \\
\ll & X^{23/20+\varepsilon}A^{-7/10+\varepsilon} \sum_{m, n\in \mathcal{R}} \psi(m) \psi(n)(nm^2)^{1/5+\varepsilon}+O\left(X^{3/4+\varepsilon}A^{-1/2+\varepsilon}|\mathcal{R}| \sum_{m\in \mathcal{R}} \psi^2(m)\right).
\end{split}  
\end{align}      
 
   We now deduce from \eqref{S1eval}, \eqref{S12}, \eqref{S12bound} and Lemma~\ref{Le1},
\begin{align*}	
\begin{split}
    |\mathcal{S}_{1}| \geq {\widehat \Phi}(1)X & \sum_{m, n \in \mathcal{R}}\psi(m)\psi(n)\sum_{\substack{ lmn^2=\text{cubic}}} \frac{1}{\sqrt{l}}\int\limits_{1}^{2}V\left(\frac{l}{A\sqrt{Xt}} \right) \dif t \\
    & +O\Bigg(X^{3/4 +\varepsilon}A^{1/2+\varepsilon}N\sum_{m\in \mathcal{R}}\psi^2(m)+X^{23/20+\varepsilon}A^{-7/10+\varepsilon} \sum_{m, n\in \mathcal{R}} \psi(m) \psi(n)(nm^2)^{1/5+\varepsilon}\Bigg).
\end{split}
\end{align*} 
   
As both $\psi$ and $V$ are both non-negative, we may further minorize the leading term in \eqref{S1exp} by keeping only the terms with $lm = n$ and discarding the rest.  Thus we obtain
\begin{align}
\label{S1remainder} 
\begin{split}
|\mathcal{S}_{1}| 
	 \ge  {\widehat \Phi}(1)X\sum_{n \in\mathcal{R}} \frac{\psi(n)}{\sqrt{n}} & \sum_{\substack{m \mid n }} \psi(m)\sqrt{m}\int\limits_{1}^{2}V\left(\frac{n}{mA\sqrt{Xt}} \right) \dif t \\
& +O\Bigg(X^{3/4 +\varepsilon}A^{1/2+\varepsilon}N\sum_{m\in \mathcal{R}}\psi^2(m)+X^{23/20+\varepsilon}A^{-7/10+\varepsilon} \sum_{m, n\in \mathcal{R}} \psi(m) \psi(n)(nm^2)^{1/5+\varepsilon}\Bigg).  
\end{split}    
\end{align}
  
 As $N \ll X$, it follows that if $m\ge n/N^{\varepsilon}$ and $A \geq 1$, $0< n/(mA\sqrt{Xt}) <1$ for $1 \leq t \leq 2$.  Applying the estimate $V(x)=1 + O\left(x^{1/2-\varepsilon}\right)$ in \eqref{2.07}, we to deduce from \eqref{S1remainder} that
\begin{align}
\label{S1lower}
\begin{split}
|\mathcal{S}_{1}| 
	 \ge {\widehat \Phi}(1)X  \sum_{n \in\mathcal{R}} \frac{\psi(n)}{\sqrt{n}} & \sum_{\substack{m\mid n\\ m\ge n/N^{\varepsilon}}} \psi(m)\sqrt{m} +O\Bigg(X^{3/4+\varepsilon}A^{-1/2+\varepsilon}\sum_{n \in\mathcal{R}} \psi(n)\sum_{\substack{m\mid n\\ m\ge n/N^{\varepsilon}}} \psi(m) \Bigg) \\ 
&  +O\Bigg(X^{3/4 +\varepsilon}A^{1/2+\varepsilon}N\sum_{m\in \mathcal{R}}\psi^2(m)+X^{23/20+\varepsilon}A^{-7/10+\varepsilon} \sum_{m, n\in \mathcal{R}} \psi(m) \psi(n)(nm^2)^{1/5+\varepsilon}\Bigg).  
\end{split}
\end{align}

   We observe from \eqref{psidef} that we have $0< \psi(p) \leq 1$ for $p \in  \mathcal{P}$. It follows from this that
\begin{align*}
\sum_{n \in\mathcal{R}} \psi(n)\sum_{\substack{m\mid n\\ m\ge n/N^{\varepsilon}}} \psi(m)  \ll \sum_{n \in\mathcal{R}} \psi(n)\sum_{\substack{ m\mid n}} \psi(m)=\sum_{m \in\mathcal{R}} \psi(m)\sum_{\substack{n \in\mathcal{R} \\ m\mid n}}\psi(n) \ll |\mathcal{R}|\sum_{m \in\mathcal{R}} \psi^2(m). 
\end{align*}  

  As  $|\mathcal{R}|\leq N$ from Lemma \ref{Le1},  we deduce from the above and \eqref{S1lower} that
\begin{align}
\label{S1lower1}
\begin{split}
|\mathcal{S}_{1}| 
	 \ge & {\widehat \Phi}(1)X\sum_{n \in\mathcal{R}} \frac{\psi(n)}{\sqrt{n}} \sum_{\substack{m\mid n\\ m\ge n/N^{\varepsilon}}} \psi(m)\sqrt{m} \\
&  +O\Bigg(X^{3/4 +\varepsilon}A^{1/2+\varepsilon}N\sum_{m\in \mathcal{R}}\psi^2(m)+X^{23/20+\varepsilon}A^{-7/10+\varepsilon} \sum_{m, n\in \mathcal{R}} \psi(m) \psi(n)(nm^2)^{1/5+\varepsilon}\Bigg). 
\end{split}
\end{align}

  Note that we have by \cite[(1.26)]{MVa1} that
\begin{align}
\label{sumk}
\begin{split}
	\sum^{[(\log_{2}N)^{\delta}]}_{k=1}\frac 1k=\log [(\log_{2}N)^{\delta}]+\gamma_0+O([(\log_{2}N)^{\delta}]^{-1}), 
\end{split}  
\end{align}     
   where $\gamma_0$ is the Euler constant.  We now deduce, from \eqref{sumk} and the definition of $\mathcal R$ given in \eqref{Rdef} that for any $n \in \mathcal R$,
\begin{align*}
\begin{split}
	n \ll \prod_{1 \leq k \leq [(\log_{2}N)^{\delta}]}(e^{k+1}\log N \log_2 N)^{a\log N/(k^2\log_3N)} \ll N^{a\delta+\varepsilon}e^{a\pi^2\log N \log_2 N/(6\log_3N)},
\end{split}  
\end{align*}   
  where the last estimation above follows by noting that $\sum^{\infty}_{k=1}1/k^2=\zeta(2)=\pi^2/6$. As $a\delta<1$, we see from the above that for any $n \in \mathcal R$, we have
\begin{align*}
\begin{split}
	n \ll N^{a\pi^2\log_2 N/(6\log_3N)+\varepsilon}.
\end{split}  
\end{align*}   
  
Using  the above, together with Cauchy's inequality, we infer
\begin{align}
\label{Sumpsitwist}
\begin{split}
\sum_{m, n\in \mathcal{R}} \psi(m) \psi(n)(nm^2)^{1/5+\varepsilon} \ll & N^{a\pi^2\log_2 N/(10\log_3N)+\varepsilon}|\mathcal R|\sum_{m \in \mathcal{R}} \psi(m)^2 \\
\ll & N^{a\pi^2\log_2 N/(10\log_3N)+\varepsilon}\sum_{m \in \mathcal{R}} \psi(m)^2,
\end{split}
\end{align}
again using the bound $|\mathcal{R}|\leq N$, furnished by Lemma \ref{Le1}. \newline
   
    It follows from \eqref{Sumpsitwist} that 
\begin{align}
\label{S1lowererror}
\begin{split}
 X^{3/4 +\varepsilon}A^{1/2+\varepsilon} & N\sum_{m\in \mathcal{R}}\psi^2(m)+X^{23/20+\varepsilon}A^{-7/10+\varepsilon} \sum_{m, n\in \mathcal{R}} \psi(m) \psi(n)(nm^2)^{1/5+\varepsilon} \\
 \ll &  \Big (X^{3/4 +\varepsilon}A^{1/2+\varepsilon}N+X^{23/20+\varepsilon}A^{-7/10+\varepsilon}N^{a\pi^2\log_2 N/(10\log_3N)+\varepsilon}\Big )\sum_{m \in \mathcal{R}} \psi(m)^2. 
\end{split}
\end{align}
The optimal choice of $A$ is
\begin{align*}
\begin{split}
 A=X^{1/3}N^{a\pi^2\log_2 N/(12\log_3N)}. 
\end{split}
\end{align*}
and renders that the right-hand side of \eqref{S1lowererror} is
\begin{align}
\label{S1lowererror1}
 \ll   X^{11/12 +\varepsilon}N^{a\pi^2\log_2 N/(24\log_3N)+\varepsilon}\sum_{m \in \mathcal{R}} \psi(m)^2. 
\end{align}

   We now apply \eqref{S1lower1}, \eqref{S1lowererror1} together with Lemmas \ref{Le2}--\ref{Le5} to see that
\begin{align}
\label{S1lower2}
\begin{split}
	 |\mathcal{S}_{1}| \ge
	\left(1+o(1)\right){\widehat \Phi}(1) & X\exp\left((\delta+ o(1))\sqrt{\frac{\log N\log_{3}N}{\log_{2}N}}\right)\left(\sum_{\substack{m\in \mathcal{R}}}\psi(m)^{2}\right) \\
& +O\left(X^{11/12 +\varepsilon}N^{a\pi^2\log_2 N/(24\log_3N)+\varepsilon}\sum_{m \in \mathcal{R}} \psi(m)^2\right).
\end{split}
\end{align}
   We now set the value of $N$ to be such that
\begin{align}
\label{Nvalue}
\begin{split}
	N^{a\pi^2\log_2 N/(24\log_3N)}=X^{1/12-5\varepsilon}. 
\end{split}
\end{align}   
  Recall that we have $1<a < 1/\delta$ so that the above choice of $N$ implies that
\begin{align}
\label{logNvalue}
\begin{split}
	\frac {\log N \log_2 N}{\log_3N}=  \frac {24}{a\pi^2}\Big(\frac 1{12}-5\varepsilon \Big)\log X >\frac {24\delta}{\pi^2}\Big(\frac 1{12}-5\varepsilon\Big)\log X. 
\end{split}
\end{align}  

  We deduce from \eqref{S1lower2}--\eqref{logNvalue} that 
\begin{align}
\label{S1lower3}
\begin{split}
	 |\mathcal{S}_{1}| \ge& \left(1+o(1)\right){\widehat \Phi}(1)X \exp\left(\left(\delta \sqrt{\frac {24\delta}{\pi^2}\Big(\frac 1{12}-5\varepsilon\Big)}+o(1)\right)\frac {\log_{3} X}{\log_{2} X}\sqrt{\log X}\right) \left(\sum_{\substack{m\in \mathcal{R}}}\psi(m)^{2}\right).
\end{split}
\end{align}

Next, we proceed to establish an upper bound of $\mathcal{S}_2$. Again, we apply Lemma \ref{PropDirpoly} and use arguments similar to those leading to \eqref{S1remainder}.  Hence,
\begin{align*}
	\mathcal{S}_{2}=& \sum_{\substack{(p,3)=1}} \ \sumstar_{\substack{\chi \shortmod{p} \\ \chi^3 = \chi_0}}(\log p)|R_{\chi}|^2\Phi \Big(\frac {p}{X} \Big)
= \sum_{m, n \in \mathcal{R}}\psi(m)\psi(n)  \sum_{\substack{(p,3)=1}} \ \sumstar_{\substack{\chi \shortmod{p} \\ \chi^3 = \chi_0}}(\log p)\chi(mn^2)\Phi \Big(\frac {p}{X} \Big) \\
	            =& {\widehat \Phi}(1)X \sum_{\substack{m, n \in \mathcal{R} \\ m n^2 =\text{cubic}}}\psi(m) \psi(n)
	             +O\Bigg(X^{1/2 + \varepsilon} \sum_{\substack{m,n \in \mathcal{R}}}\psi(m) \psi(n)\log
  (mn^2+2)\Bigg)\\
	            =& {\widehat \Phi}(1)X	\sum_{\substack{n \in \mathcal{R}}}\psi(n)^{2} + O\left(X^{1/2 +\varepsilon} |\mathcal{R}| \sum_{\substack{m\in  \mathcal{R}}} \psi(m)^{2} \right),
\end{align*}
 where the last equality above follows by noting that $mn^2$ is a cube only if $m=n$ since $m$ and $n$ are both square-free.\newline

Again, using $|\mathcal{R}|\leq N= X^{o(1)}$, we deduce from the above that 
\begin{align}    
\label{Sbound}        
  \mathcal{S}_{2} \le \left(1+o(1)\right){\widehat \Phi}(1)X
	            \left(\sum_{\substack{m\in \mathcal{R}}}\psi(m)^{2}\right).
\end{align}
Hence, for sufficiently large $X$ and arbitrary small $\varepsilon>0$, we deduce from \eqref{maxlower}, \eqref{S1lower3} and \eqref{Sbound} that 
 \begin{align*}
\max_{\substack{X< p \le 2X \\\chi \shortmod{p} \\ \chi^3 = \chi_0 }}\Big|L(\tfrac 12, \chi)\Big | &\geq \frac{|\mathcal{S}_1|}{\mathcal{S}_2}\geq \exp\left(\left(\delta \sqrt{\frac {24\delta}{\pi^2}\Big(\frac 1{12}-5\varepsilon \Big)}+o(1)\right)\frac {\log_{3} X}{\log_{2} X}\sqrt{\log X}\right).
\end{align*}
With $\delta \rightarrow 1^-$, we obtain the desired estimation for the first expression given in  \eqref{cubicbound}. 

\subsection{Proof of \eqref{quarticbound}}

As the proof is similar to that given in the previous section, we shall be brief here.  Keeping the notation in the previous section, set
\begin{align*} 
 \mathcal{T}_{1}:=\sum_{\substack{(p,2)=1}} \ \sumstar_{\substack{\chi \shortmod{p} \\ \chi^4 = \chi_0}}  (\log p)L(\tfrac{1}{2},\chi) |R_{\chi}|^2 \Phi \Big(\frac {p}{X}\Big) \quad \mbox{and} \quad \mathcal{T}_{2}:=\sum_{\substack{(p,2)=1}} \ \sumstar_{\substack{\chi \shortmod{p} \\ \chi^4 = \chi_0}}(\log p)|R_{\chi}|^2\Phi \Big(\frac {p}{X}\Big).
\end{align*}
  Then we have 
\begin{align}  
\label{maxlowerquartic}
 \max_{\substack{X< p \le 2X \\\chi \shortmod{p} \\ \chi^4 = \chi_0 }}\Big|L(\tfrac 12, \chi)\Big |\ge \frac{|\mathcal{T}_1|}{\mathcal{T}_2}. 
\end{align}
We first apply \eqref{approxfuncquartic}, \eqref{approxfuncquarticprime} and the expression $(1\pm \chi(-1))/2$ to detect those characters satisfying $\chi(-1)=\pm 1$  to see that 
\begin{align}
\label{S1evalquartic}
\begin{split}
	\mathcal{T}_{1}=& T_{1,1, +}+T_{1,1, -}+T_{1,2, +}+T_{1,2, -},   
\end{split}  
\end{align}
   where for some constant $B=B(X)>0$, 
\begin{align*}
\begin{split}
	\mathcal{T}_{1,1, \pm}=& \sum_{m, n\in \mathcal{R}} \psi(m) \psi(n)\sum_{l\geq 1 } \frac{1}{\sqrt{l}}\sum_{\substack{(p,2)=1}} \ \sumstar_{\substack{\chi \shortmod{p} \\ \chi^4 = \chi_0}}  (\log p)\frac {1\pm \chi(-1)}{2}\chi(lm)\overline{\chi}(n)V_{\pm}\left(\frac{l}{B\sqrt{p}} \right)\Phi \Big(\frac {p}{X}\Big), \quad \mbox{and} \\
	\mathcal{T}_{1,2, \pm}=& \pm \sum_{m, n\in \mathcal{R}} \psi(m) \psi(n)\sum_{l\geq 1 } \frac{1}{\sqrt{l}} \\
& \hspace*{1cm} \times \sum_{\substack{ \varpi \odd \\ \varpi \notin \mz }}\frac {1\pm \chi_{\varpi}(-1)}{2} \cdot \frac {(\log N(\varpi))\chi_{\varpi}(m)\overline\chi_{\varpi}((-2i)^3ln)g_4(\varpi)}{\sqrt{N(\varpi)}}V_{\pm}\left(\frac{lB}{\sqrt{N(\varpi)}} \right)\Phi\left(\frac{N(\varpi)}{X}\right) .   
\end{split}  
\end{align*}   

Similar to \eqref{S1exp},
\begin{align}	
\label{S1expquartic}
\begin{split}
 \mathcal{T}_{1,1, +}+\mathcal{T}_{1,1, -}  = \frac 12{\widehat \Phi}(1)X \sum_{m, n \in \mathcal{R}}\psi(m)\psi(n) & \sum_{\substack{ lmn^3=\text{quartic}}} \frac{1}{\sqrt{l}}\int\limits_{1}^{2}\left(V_{+}\left(\frac{l}{B\sqrt{Xt}} \right)+V_{-}\left(\frac{l}{B\sqrt{Xt}} \right)\right) \dif t \\
 & +O\Bigg(X^{3/4 +\varepsilon}B^{1/2+\varepsilon}|\mathcal{R}|\sum_{m\in \mathcal{R}}\psi^2(m)\Bigg).
\end{split}
\end{align}
We apply Proposition \ref{lemg3}, analogous to \eqref{S12bound}, getting
\begin{align}
\label{S12boundquartic}
\begin{split}
T_{1,2, +}+T_{1,2, -} \ll & \sum_{m, n\in \mathcal{R}} \psi(m) \psi(n)\sum_{l\leq (\sqrt{X}/B)^{1+\varepsilon}} \frac{X^{7/8+\varepsilon}}{\sqrt{l}}(lnm^3)^{1/8+\varepsilon}+O\left(X^{3/4+\varepsilon}B^{-1/2+\varepsilon}|\mathcal{R}| \sum_{m\in \mathcal{R}} \psi^2(m)\right) \\
\ll & X^{19/16+\varepsilon}B^{-5/8+\varepsilon} \sum_{m, n\in \mathcal{R}} \psi(m) \psi(n)(nm^3)^{1/8+\varepsilon}+O\left(X^{3/4+\varepsilon}B^{-1/2+\varepsilon}|\mathcal{R}| \sum_{m\in \mathcal{R}} \psi^2(m)\right).
\end{split}  
\end{align}      
 
From \eqref{S1evalquartic}--\eqref{S12boundquartic} and Lemma \ref{Le1}, we decude that
\begin{align}	
\label{S1exp1quartic}
\begin{split}
    |\mathcal{T}_{1}| \geq  \frac 12{\widehat \Phi}(1)X & \sum_{m, n \in \mathcal{R}}\psi(m)\psi(n)\sum_{\substack{ lmn^3=\text{quartic}}} \frac{1}{\sqrt{l}}\int\limits_{1}^{2}\left(V_{+}\left(\frac{l}{B\sqrt{Xt}} \right)+V_{-}\left(\frac{l}{B\sqrt{Xt}} \right)\right) \dif t  \\
    & +O\Bigg(X^{3/4 +\varepsilon}B^{1/2+\varepsilon}N\sum_{m\in \mathcal{R}}\psi^2(m)+X^{19/16+\varepsilon}B^{-5/8+\varepsilon} \sum_{m, n\in \mathcal{R}} \psi(m) \psi(n)(nm^3)^{1/8+\varepsilon}\Bigg).  
\end{split}
\end{align} 
  
The positivity of $\psi$ and $V$ again allows us to keep only the terms with $lm=n$ in the main term of \eqref{S1exp1quartic}.  Hence, similar to \eqref{S1lower1}, 
\begin{align}
\label{S1lower1quartic}
\begin{split}
|\mathcal{T}_{1}| 
	 \ge {\widehat \Phi}(1)X & \sum_{n \in\mathcal{R}} \frac{\psi(n)}{\sqrt{n}} \sum_{\substack{m\mid n\\ m\ge n/N^{\varepsilon}}} \psi(m)\sqrt{m} \\
&  +O\Bigg(X^{3/4 +\varepsilon}B^{1/2+\varepsilon}N\sum_{m\in \mathcal{R}}\psi^2(m)+X^{19/16+\varepsilon}B^{-5/8+\varepsilon}  \sum_{m, n\in \mathcal{R}} \psi(m) \psi(n)(nm^3)^{1/8+\varepsilon}\Bigg). 
\end{split}
\end{align}

Similar to \eqref{Sumpsitwist},
\begin{align}
\label{Sumpsitwistquartic}
\begin{split}
	 \sum_{m, n\in \mathcal{R}} \psi(m) \psi(n)(nm^3)^{1/8+\varepsilon} \ll & N^{a\pi^2\log_2 N/(12\log_3N)+\varepsilon}\sum_{m \in \mathcal{R}} \psi(m)^2.
\end{split}
\end{align}   
  
    It follows from \eqref{Sumpsitwistquartic} that 
\begin{align}
\label{S1lowererrorquartic}
\begin{split}
 X^{3/4 +\varepsilon}B^{1/2+\varepsilon} & N\sum_{m\in \mathcal{R}}\psi^2(m)+X^{19/16+\varepsilon}B^{-5/8+\varepsilon}\sum_{m, n\in \mathcal{R}} \psi(m) \psi(n)(nm^3)^{1/8+\varepsilon} \\
 \ll &  \Big (X^{3/4 +\varepsilon}B^{1/2+\varepsilon}N+X^{19/16+\varepsilon}B^{-5/8+\varepsilon} N^{a\pi^2\log_2 N/(12\log_3N)+\varepsilon}\Big )\sum_{m \in \mathcal{R}} \psi(m)^2. 
\end{split}
\end{align}
  We now optimize our choice of $B$ as
\begin{align*}
\begin{split}
 B=X^{7/18}N^{2a\pi^2\log_2 N/(27\log_3N)}. 
\end{split}
\end{align*}
When inserted into \eqref{S1lowererrorquartic}, this yields 
\begin{align}
\label{S1lowererror1quartic}
\begin{split}
 X^{3/4 +\varepsilon}A^{1/2+\varepsilon}N  \sum_{m\in \mathcal{R}}\psi^2(m)  +X^{19/16+\varepsilon}B^{-5/8+\varepsilon} & \sum_{m, n\in \mathcal{R}} \psi(m) \psi(n)(nm^3)^{1/8+\varepsilon} \\
 \ll &  X^{17/18 +\varepsilon}N^{a\pi^2\log_2 N/(27\log_3N)+\varepsilon}\sum_{m \in \mathcal{R}} \psi(m)^2. 
\end{split}
\end{align}

   We now apply \eqref{S1lower1quartic}, \eqref{S1lowererror1quartic} together with Lemmas \ref{Le2}--\ref{Le5} to see that
\begin{align}
\label{S1lower2quartic}
\begin{split}
	 |\mathcal{T}_{1}| \ge  \left(1+o(1)\right){\widehat \Phi}(1) X \exp & \left((\delta+ o(1))\sqrt{\frac{\log N\log_{3}N}{\log_{2}N}}\right)\left(\sum_{\substack{m\in \mathcal{R}}}\psi(m)^{2}\right) \\
& +O\left(X^{17/18 +\varepsilon}N^{a\pi^2\log_2 N/(27\log_3N)+\varepsilon}\sum_{m \in \mathcal{R}} \psi(m)^2\right).
\end{split}
\end{align}
We now choose the value of $N$ satisfying
\begin{align}
\label{Nvaluequartic}
\begin{split}
	N^{a\pi^2\log_2 N/(27\log_3N)}=X^{1/18-5\varepsilon}. 
\end{split}
\end{align}   
  Recall that we have $1<a < 1/\delta$ so that the above implies that
\begin{align}
\label{logNvaluequartic}
\begin{split}
	\frac {\log N \log_2 N}{\log_3N}=  \frac {27}{a\pi^2}\Big (\frac 1{18}-5\varepsilon \Big)\log X >\frac {27\delta}{\pi^2}\Big(\frac 1{18}-5\varepsilon\Big)\log X. 
\end{split}
\end{align}  

  We deduce from \eqref{S1lower2quartic}--\eqref{logNvaluequartic} that 
\begin{align}
\label{S1lower3quartic}
\begin{split}
	 |\mathcal{T}_{1}| 
		 \ge  &\left(1+o(1)\right){\widehat \Phi}(1)X \exp\left(\left(\delta \sqrt{\frac {27\delta}{\pi^2}\Big(\frac 1{18}-5\varepsilon\Big)}+o(1)\right)\frac {\log_{3} X}{\log_{2} X}\sqrt{\log X}\right) \left(\sum_{\substack{m\in \mathcal{R}}}\psi(m)^{2}\right).
\end{split}
\end{align}

Next, similar to \eqref{Sbound}, we have
\begin{align*}            
  \mathcal{T}_{2} \le \left(1+o(1)\right){\widehat \Phi}(1)X
	            \left(\sum_{\substack{m\in \mathcal{R}}}\psi(m)^{2}\right).
\end{align*}
Hence, for sufficiently large $X$ and arbitrary small $\varepsilon>0$, \eqref{maxlowerquartic}, \eqref{S1lower3quartic} and the above render
 \begin{align*}
\max_{\substack{X< p \le 2X \\\chi \shortmod{p} \\ \chi^4 = \chi_0 }}\Big|L(\tfrac 12, \chi)\Big | &\geq \frac{|\mathcal{T}_1|}{\mathcal{T}_2}\geq \exp\left(\left(\delta \sqrt{\frac {27\delta}{\pi^2}\Big(\frac 1{18}-5\varepsilon\Big)}+o(1)\right)\frac {\log_{3} X}{\log_{2} X}\sqrt{\log X}\right).
\end{align*}
 By taking $\delta \rightarrow 1^-$, we arrive at the desired bound in \eqref{quarticbound}. 
   
\hspace{0.1in}

\noindent{\bf Acknowledgments.} P. G. is supported in part by NSFC grant 12471003 and L. Z. by the FRG Grant PS71536 at the University of New South Wales.

\bibliography{biblio}

\begin{bibdiv}
\begin{biblist}

\bib{Aistleitner16}{article}{
      author={Aistleitner, C.},
       title={Lower bounds for the maximum of the {R}iemann zeta function along
  vertical lines},
        date={2016},
     journal={Math. Ann.},
      volume={365},
      number={1-2},
       pages={473\ndash 496},
}

\bib{B&Y}{article}{
      author={Baier, S.},
      author={Young, M.~P.},
       title={Mean values with cubic characters},
        date={2010},
     journal={J. Number Theory},
      volume={130},
      number={4},
       pages={879\ndash 903},
}

\bib{BEW}{book}{
      author={Berndt, B.~C.},
      author={Evans, R.~J.},
      author={Williams, K.~S.},
       title={{G}auss and {J}acobi sums},
      series={Canadian Mathematical Society Series of Monographs and Advanced
  Texts},
   publisher={John Wiley \& Sons},
     address={New York},
        date={1998},
}

\bib{BS18}{article}{
      author={Bondarenko, A.},
      author={Seip, K.},
       title={Extreme values of the {R}iemann zeta function and its argument},
        date={2018},
     journal={Math. Ann.},
      volume={372},
      number={3-4},
       pages={999\ndash 1015},
}

\bib{DM25}{article}{
      author={Darbar, P.},
      author={Maiti, G.},
       title={Large values of quadratic {D}irichlet {$L$}-functions},
        date={2025},
     journal={Math. Ann.},
      volume={392},
      number={4},
       pages={4573\ndash 4605},
}

\bib{DG22}{article}{
      author={David, C.},
      author={G\"{u}lo\u{g}lu, A.~M.},
       title={One-level density and non-vanishing for cubic {$L$}-functions
  over the {E}isenstein field},
        date={2022},
     journal={Int. Math. Res. Not. IMRN},
      number={23},
       pages={18833\ndash 18873},
}

\bib{BT19}{article}{
      author={de~la Bret\`eche, R.},
      author={Tenenbaum, G.},
       title={Sommes de {G}\'al et applications},
        date={2019},
     journal={Proc. Lond. Math. Soc. (3)},
      volume={119},
      number={1},
       pages={104\ndash 134},
}

\bib{Diac}{article}{
      author={Diaconu, A.},
       title={Mean square values of {H}ecke {$L$}-series formed with {$r$}-th
  order characters},
        date={2004},
     journal={Invent. Math.},
      volume={157},
      number={3},
       pages={635\ndash 684},
}

\bib{FHX26}{article}{
      author={Fan, M.},
      author={Hua, S.},
      author={Xie, S.},
       title={Extreme central values of quadratic {D}irichlet {$L$}-functions
  with prime conductors},
        date={2026},
     journal={Q. J. Math.},
      volume={77},
      number={1},
       pages={175\ndash 199},
}

\bib{Gao2026-6}{article}{
      author={Gao, P.},
       title={Large values of quadratic {D}irichlet {$L$}-functions of
  prime-related moduli},
        date={Preprint},
        note={arXiv:2606.15635},
}

\bib{G&Zhao2020-1}{article}{
      author={Gao, P.},
      author={Zhao, L.},
       title={Mean values of cubic and quartic {D}irichlet characters},
        date={2020},
     journal={Funct. Approx. Comment. Math.},
      volume={63},
      number={2},
       pages={227\ndash 245},
}

\bib{G&Zhao7}{article}{
      author={Gao, P.},
      author={Zhao, L.},
       title={Moments of central values of quartic {D}irichlet
  {$L$}-functions},
        date={2021},
     journal={J. Number Theory},
      volume={228},
       pages={342\ndash 358},
}

\bib{G&Zhao22-1}{article}{
      author={Gao, P.},
      author={Zhao, L.},
       title={Bounds for moments of cubic and quartic {D}irichlet
  {$L$}-functions},
        date={2022},
     journal={Indag. Math. (N.S.)},
      volume={33},
      number={6},
       pages={1263\ndash 1296},
}

\bib{G&Zhao24-02}{article}{
      author={Gao, P.},
      author={Zhao, L.},
       title={Ratios conjecture of quartic -functions of prime moduli},
        date={Preprint},
        note={arXiv:2402.17198},
}

\bib{I&R}{book}{
      author={Ireland, K.},
      author={Rosen, M.},
       title={{A} {C}lassical {I}ntroduction to {M}odern {N}umber {T}heory},
     edition={Second edition},
      series={Graduate Texts in Mathematics},
   publisher={Springer-Verlag},
     address={New York},
        date={1990},
      volume={84},
}

\bib{iwakow}{book}{
      author={Iwaniec, H.},
      author={Kowalski, E.},
       title={{A}nalytic {N}umber {T}heory},
      series={American Mathematical Society Colloquium Publications},
   publisher={American Mathematical Society},
     address={Providence},
        date={2004},
      volume={53},
}

\bib{Lemmermeyer}{book}{
      author={Lemmermeyer, F.},
       title={{R}eciprocity laws. {F}rom {E}uler to {E}isenstein},
   publisher={Springer-Verlag, Berlin},
        date={2000},
}

\bib{MVa1}{book}{
      author={Montgomery, H.~L.},
      author={Vaughan, R.~C.},
       title={{M}ultiplicative number theory. {I}. {C}lassical theory},
      series={Cambridge Studies in Advanced Mathematics},
   publisher={Cambridge University Press},
     address={Cambridge},
        date={2007},
      volume={97},
}

\bib{P}{article}{
      author={Patterson, S.~J.},
       title={The distribution of general {G}auss sums and similar arithmetic
  functions at prime arguments},
        date={1987},
     journal={Proc. London Math. Soc. (3)},
      volume={54},
       pages={193\ndash 215},
}

\bib{Radziwill&Sound}{article}{
      author={Radziwi{\l \l}, M.},
      author={Soundararajan, K.},
       title={Moments and distribution of central {$L$}-values of quadratic
  twists of elliptic curves},
        date={2015},
     journal={Invent. Math.},
      volume={202},
      number={3},
       pages={1029\ndash 1068},
}

\bib{sound1}{article}{
      author={Soundararajan, K.},
       title={Nonvanishing of quadratic {D}irichlet {$L$}-functions at
  $s=\frac{1}{2}$},
        date={2000},
     journal={Ann. of Math. (2)},
      volume={152},
      number={2},
       pages={447\ndash 488},
}

\bib{Sound08}{article}{
      author={Soundararajan, K.},
       title={Extreme values of zeta and {$L$}-functions},
        date={2008},
     journal={Math. Ann.},
      volume={342},
      number={2},
       pages={467\ndash 486},
}

\end{biblist}
\end{bibdiv}
\bibliographystyle{amsxport}

\vspace*{.5cm}

\end{document}